\documentclass[11pt,oneside]{article}
\usepackage[centertags]{amsmath}
\usepackage{amssymb}
\usepackage{amsthm}
\usepackage{amsfonts}
\usepackage{graphicx}
\providecommand{\U}[1]{\protect\rule{.1in}{.1in}}
\newtheorem{theorem}{Theorem}[section]
\newtheorem{corollary}{Corollary}[section]
\newtheorem{remark}{Remark}[section]

\newtheorem{proposition}{Proposition}[section]
\newtheorem{definition}{Definition}[section]

\newcommand{\punt}{\boldsymbol{.}}

\begin{document}

\title{Symbolic solutions of some linear recurrences}
\author{E. Di Nardo \thanks{%
Dipartimento di Matematica e Informatica, Universit\`a degli Studi
della Basilicata, Viale dell'Ateneo Lucano 10, 85100 Potenza,
Italia, elvira.dinardo@unibas.it}, D. Senato
\thanks{Dipartimento di Matematica e Informatica, Universit\`a degli Studi della
Basilicata, Viale dell'Ateneo Lucano 10, 85100 Potenza, Italia,
domenico.senato@unibas.it}}
\date{ }
\maketitle

\begin{abstract}
A symbolic method for solving linear recurrences of
combinatorial and statistical interest is introduced. This method 
essentially relies on a representation of polynomial 
sequences as moments of a symbol that looks as the framework of a random variable
with no reference to any probability space. We give several examples of applications and state an 
explicit form for the class of linear recurrences involving Sheffer 
sequences satisfying a special initial condition. The results 
here presented can be easily implemented in a symbolic software.
\end{abstract}

\textsf{\textbf{keywords}: linear recurrences, Sheffer sequences,
classical umbral calculus, Dyck paths. }\newline

\section{Introduction}

Many counting problems, especially those related to lattice path counting and 
random walks and ballot paths \cite{Narayana}, give rise to recurrence relations.  Once a recursion has been
established, Sheffer polynomials are often a simple and general tool for
finding answers in closed form. Main contributions in this respect are due
to Niederhausen \cite{Nie0,Nie1,Nie2,Nie3}. Further contributions are given by
Razpet \cite{Razpet} and Di Bucchianico and Soto y Koelemeijer \cite
{dibucchianico1}.  Most of the recent papers involving Sheffer
sequences still use the language of finite linear operators
\cite{WW,Nie4}. In this paper we propose a symbolic theory of Sheffer
sequences, developed in \cite{DNS}, in which their main properties are
encoded by using a symbolic method arising from the notion of
umbra, introduced by Rota and Taylor in \cite{SIAM} and developed
in \cite{dinardoeurop}. As a matter of fact,
an umbra looks as the framework of a random variable with no reference to
any probability space, someway getting closer to statistical methods. 
This method has given rise to very efficient algorithms for computing 
unbiased estimators such as $k$-statistics and polykays \cite{fast} and
generalizations of Sheppard's corrections \cite{Sheppard}.  

An umbral calculus consists of a set $A=\{\alpha,\beta, \ldots\},$ called
the \textit{alphabet}, whose elements are named \textit{umbrae} and a linear
functional $E,$ called \textit{evaluation}, defined on the polynomial ring $%
R[A]$ and taking values in $R,$ where $R$ is a suitable ring. The linear
functional $E$ is such that $E[1]=1$ and
\begin{equation*}
E[\alpha^{i} \beta^{j} \cdots\gamma^{k}] = E[\alpha ^{i}]E[\beta^{j}] \cdots
E[\gamma^{k}], \quad {\hbox{(uncorrelation property)}}
\end{equation*}
for any set of distinct umbrae in $A$ and for $i,j,\ldots,k$ nonnegative
integers.

A sequence $a_{0}=1,a_{1},a_{2}, \ldots$ in $R$ is umbrally represented by
an umbra $\alpha$ when $E[\alpha^{n}]=a_{n}$ for all nonnegative integers $n.
$ The elements $\{a_{n}$\} are called \textit{moments} of the umbra $\alpha.$
Special umbrae are:

\textit{i)} the \textit{augmentation} umbra $\epsilon\in A,$ such that $%
E[\epsilon^{n}] = \delta_{0,n},$ for any nonnegative integer $n;$

\textit{ii)} the \textit{unity} umbra $u \in A,$ such that $E[u^{n}]=1,$ for
any nonnegative integer $n;$

\textit{iii)} the \textit{singleton} umbra $\chi \in A,$ such that $%
E[\chi^{1}] = 1$ and $E[\chi^{n}] = 0,$ for any nonnegative integer $n \geq
2;$

\textit{iv)} the \textit{Bell} umbra $\beta \in A,$ such that such that $%
E[\beta^n] = B_n,$ for any nonnegative integers $n,$ where $\{B_n\}$ are the
Bell numbers;

\textit{v)} the \textit{Bernoulli} umbra $\iota \in A,$ such that $%
E[\iota^n] = {\mathcal{B}}_n,$ for any nonnegative integers $n,$ where $\{%
\mathcal{B}_n\}$ are the Bernoulli numbers.

An umbral polynomial is a polynomial $p \in R[A].$ The support of $p$ is the
set of all umbrae occurring in $p.$ If $p$ and $q$ are two umbral
polynomials then $p$ and $q$ are \textit{uncorrelated} if and only if their
supports are disjoint. Moreover, $p$ and $q$ are \textit{umbrally equivalent}
if and only if $E[p]=E[q],$ in symbols $p\simeq q.$

So, in this new setting, to which we refer as the classical umbral
calculus, there are two basic devices. The first one is to
represent a unital sequence of scalars or polynomials by a symbol
$\alpha ,$ called an umbra. In other words, the
sequence $1,a_{1},a_{2},\ldots $ is represented by means of the sequence $%
1,\alpha ,\alpha ^{2},\ldots $ of powers of $\alpha $ via an
operator $E,$ resembling the expectation operator of random
variables. The second device is to represent the same sequence
$1,a_{1},a_{2},\ldots $ by means of distinct umbrae, as it could
happen also in probability theory with independent and identically
distributed random variables. More precisely, two umbrae $\alpha $
and $\gamma $ are \textit{similar} when $\alpha ^{n}$ is umbrally
equivalent to $\gamma ^{n},$ for all nonnegative integer $n,$ in
symbols $\alpha \equiv \gamma \Leftrightarrow \alpha ^{n}\simeq
\gamma ^{n}$ for all $n\geq 0$. It is mainly thanks to these devices that the early umbral calculus \cite%
{SIAM} has had a rigorous and simple formal look.

 Section 2 of this paper is devoted to resume Sheffer umbrae
and their characterizations. Let us underline that a complete
version of the theory of Sheffer polynomials by means of umbrae
can be find in \cite{DNS}. Here, we have chosen to recall
terminology, notation and the basic definitions strictly necessary
to deal with the object of this paper. Section 3 involves linear
recurrence relations whose solutions can be expressed via Sheffer
sequences. Two of the examples we propose involve Fibonacci
numbers. The first example encodes a definite integral by means of
a suitable substitution of a special umbra. The last example is
related to counting patters avoiding ballot paths which has a
connection with Dyck paths \cite{Sapounakis}.


\section{Sheffer umbrae.}


\subparagraph{Auxiliary umbrae.}

Thanks to the notion of similar umbrae, the alphabet $A$ has been extended
with the so-called \textit{auxiliary} umbrae, resulting from operations
among similar umbrae . This leads to the construction of a saturated umbral
calculus \cite{SIAM}, in which auxiliary umbrae are handled as elements of
the alphabet. Two special auxiliary umbrae are:

\textit{i)} the \textit{dot-product} of $n$ and $\alpha$, denoted by the
symbol $n \boldsymbol{.} \alpha,$ similar to the sum $\alpha^{\prime
}+\alpha^{\prime\prime}+ \cdots+ \alpha^{\prime\prime\prime},$ where $%
\{\alpha^{\prime},\alpha^{\prime\prime},\ldots,\alpha^{\prime\prime\prime}\}$
is a set of $n$ distinct umbrae, each one similar to the umbra $\alpha;$

\textit{ii)} the inverse of an umbra $\alpha,$ denoted by the symbol $-1
\boldsymbol{.} \alpha,$ such that $-1 \boldsymbol{.} \alpha + \alpha \equiv
\epsilon.$

The computational power of the umbral syntax can be improved through the
construction of new auxiliary umbrae by means of the notion of generating
function of umbrae. The formal power series
\begin{equation}
u + \sum_{n \geq1} \alpha^{n} \frac{t^{n}}{n!}   \label{(gf)}
\end{equation}
is the \textit{generating function} of the umbra $\alpha,$ and it is denoted
by $e^{\alpha t}.$ The notion of umbral equivalence and similarity can be
extended coefficientwise to formal power series (\ref{(gf)}), that is $%
\alpha\equiv\beta\Leftrightarrow e^{\alpha t} \simeq e^{\beta t}$ (see \cite%
{Taylor1} for a formal construction). Note that any exponential formal power
series $f(t) = 1 + \sum_{n \geq1} a_{n} t^{n} / n!$ can be umbrally
represented by a formal power series (\ref{(gf)}). In fact, if the sequence $%
1,a_{1},a_{2},\ldots$ is umbrally represented by $\alpha$ then $%
f(t)=E[e^{\alpha t}],$ that is $f(t) \simeq e^{\alpha t},$ assuming that we
extend $E$ by linearity. In this case, instead of $f(t)$ we write $%
f(\alpha,t)$ and we say that $f(\alpha,t)$ is umbrally represented by $%
\alpha.$ Henceforth, when no confusion occurs, we just say that $f(\alpha,t)$
is the generating function of $\alpha.$ For example for the augmentation
umbra we have $f(\epsilon,t)=1,$ for the unity umbra we have $f(u,t)=e^{t},$
and for the singleton umbra we have $f(\chi,t)=1+t.$ Moreover for the
Bernoulli umbra we have $f(\iota,t)=t/(e^{t}-1)$ (see \cite{SIAM}) and for
the Bell umbra we have $f(\beta,t)=\exp(e^{t}-1)$ (see \cite{dinardoeurop}).

The advantage of an umbral notation for generating functions is the
representation of operations among generating functions through symbolic
operations among umbrae. For example, the product of exponential generating
functions is umbrally represented by a sum of the corresponding umbrae:
\begin{equation}
f(\alpha,t)\,f(\gamma,t)\simeq e^{(\alpha+\gamma)t}\quad\hbox{with}\quad
f(\alpha,t)\simeq e^{\alpha t}\,\,\hbox{and}\,\,f(\gamma,t)\simeq e^{\gamma
t}.   \label{(summation)}
\end{equation}
Then for the inverse of an umbra we have $f(-1 \boldsymbol{.}
\alpha,t) = 1 / f(\alpha,t)$ and for the dot-product of $n$ and
$\alpha$ we have $f(n \boldsymbol{.} \alpha,t) = f(\alpha,t)^n.$
Suppose we replace $R$ by a suitable polynomial ring having
coefficients in $R$ and any
desired number of indeterminates. Then, an umbra is said to be \textit{scalar%
} if the moments are elements of $R$ while it is said to be \textit{%
polynomial} if the moments are polynomials. In this paper, we deal with $%
R[x,y].$ We can define the dot-product of $x$ and $\alpha $ as an auxiliary
umbra having generating functions $f(x\boldsymbol{.}\alpha ,t)=f(\alpha
,t)^{x}.$ This construction has been done formally in \cite{dinardoeurop} where
the moments of $x\boldsymbol{.}\alpha $ have been also computed. A feature
of the classical umbral calculus is the construction of new auxiliary umbrae
by suitable symbolic substitutions. In $n\boldsymbol{.}\alpha $ replace the
integer $n$ by an umbra $\gamma $. The auxiliary umbra $\gamma \boldsymbol{.}%
\alpha $ is the \textit{dot-product} of $\alpha $ and $\gamma $ having
generating function $f(\gamma \boldsymbol{.}\alpha ,t) = f(\gamma ,\log
f(\alpha ,t)).$ If we replace the umbra $\gamma $ with the Bell umbra $\beta
,$ we have $f(\beta \boldsymbol{.}\alpha ,t) = \exp [f(\alpha ,t)-1],$
so that the umbra $\beta \boldsymbol{.}\alpha $ has been called $\alpha $%
-partition umbra. The $\alpha $-partition umbra plays a crucial role in the
umbral representation of the composition of exponential generating
functions. Indeed, first we can consider the polynomial $\alpha $-partition
umbra $x\boldsymbol{.}\beta \boldsymbol{.}\alpha $ with generating function $%
f(x\boldsymbol{.}\beta \boldsymbol{.}\alpha ,t)=\exp [x(f(\alpha ,t)-1)].$
Its moments are \cite{dinardoeurop}
\begin{equation}
E[(x\boldsymbol{.}\beta \boldsymbol{.}\alpha
)^{i}]=\sum_{j=1}^{i}x^{j}\,B_{i,j}(a_{1},a_{2},\ldots ,a_{i-j+1}),
\label{(gr:3ter)}
\end{equation}%
$B_{i,j}$ are the (partial) Bell exponential polynomials \cite{Riordan} and $%
a_{i}$ are the moments of the umbra $\alpha .$ When $x$ is replaced by an
umbra $\gamma ,$ the moments in (\ref{(gr:3ter)}) give the coefficients of
the composition of the generating functions $f(\alpha ,t)$ and $f(\gamma ,t),
$ that is $f(\gamma \boldsymbol{.}\beta \boldsymbol{.}\alpha ,t)=f[\gamma
,f(\alpha ,t)-1].$ The umbra $\gamma \boldsymbol{.}\beta \boldsymbol{.}%
\alpha $ is the \textit{composition umbra} of $\alpha $ and $\gamma .$ If $%
E[\alpha ]\neq 0,$ the umbra $\alpha ^{\scriptscriptstyle<-1>}$ is the
compositional inverse of $\alpha ,$ with generating function $f(\alpha ^{%
\scriptscriptstyle<-1>},t)=f^{\scriptscriptstyle<-1>}(\alpha ,t),$ where $f^{%
\scriptscriptstyle<-1>}(\alpha ,t)$ is the compositional inverse of $%
f(\alpha ,t)$ that is $f[\alpha ^{\scriptscriptstyle<-1>},f(\alpha
,t)-1]=f[\alpha ,f(\alpha ^{\scriptscriptstyle<-1>},t)-1]=1+t.$

\subparagraph{Sheffer umbrae.}

In the following let $\alpha$ and $\gamma$ be scalar umbrae, with $%
g_{1}=E[\gamma] \ne0.$

\begin{definition}
\label{(defshef)} A polynomial umbra $\sigma_x$ is said to be a Sheffer
umbra for $(\alpha,\gamma)$ if $\sigma_x \equiv \alpha + x \boldsymbol{.}
\beta \boldsymbol{.} \gamma^{\scriptscriptstyle <-1>}.$
\end{definition}

In the following, we denote by $\gamma^*$ the $\gamma^{\scriptscriptstyle %
<-1>}$-partition umbra $\beta \boldsymbol{.} \gamma^{\scriptscriptstyle %
<-1>}.$ The umbra $\gamma^*$ is called the \textit{adjoint} umbra of $\gamma$
and in particular $f(\gamma^{*}, t) = \exp[f^{\scriptscriptstyle %
<-1>}(\gamma,t)-1].$ The name parallels the adjoint of an umbral operator
\cite{Roman} since $\gamma\boldsymbol{.} \alpha^{*}$ gives the umbral
composition of $\gamma$ and $\alpha^{\scriptscriptstyle <-1>}.$ So a
polynomial umbra $\sigma_x$ is said to be a Sheffer umbra for $%
(\alpha,\gamma)$ if
\begin{equation}
\sigma_x \equiv \alpha+ x \boldsymbol{.} \gamma^{*}.   \label{(conrep)}
\end{equation}
In the following, we denote a Sheffer umbra by $\sigma_{x}^{(\alpha,\gamma)}$
in order to make explicit the dependence on $\alpha$ and $\gamma.$ From (\ref%
{(conrep)}), the generating function of $\sigma_{x}^{(\alpha,\gamma)}$ is
\begin{equation}
f[\sigma^{(\alpha,\gamma)}_{x}, t] = f(\alpha,t) \, e^{x \, [f^{%
\scriptscriptstyle <-1>}(\gamma,t)-1]}.   \label{gf}
\end{equation}
Note that, given the umbra $\gamma,$ any Sheffer umbra is uniquely
determined by its moments evaluated at $0,$ since via equivalence (\ref%
{(conrep)}) we have $\sigma_{0}^{(\alpha,\gamma)}\equiv \alpha.$

Among Sheffer umbrae, a special role is played by the \textit{associated}
umbra. Let us consider a Sheffer umbra for the umbrae $(\epsilon,\gamma),$
where $\gamma$ has a compositional inverse and $\epsilon$ is the augmentation
umbra.

\begin{definition}
\label{(th1bis)} The polynomial umbra $\sigma_x \equiv x \boldsymbol{.}
\gamma^{*},$ where $\gamma^{*}$ is the adjoint umbra of $\gamma,$ is the
associated umbra of $\gamma.$
\end{definition}

The associated umbrae are polynomial umbrae whose moments $\{p_{n}(x)\}$
satisfy the well-known binomial identity
\begin{equation}
p_{n}(x+y) = \sum_{k=0}^{n} \left(
\begin{array}{c}
n \\
k \\
\end{array}
\right) \, p_{k}(x) \, p_{n-k}(y)   \label{(binom1)}
\end{equation}
for all $n = 0,1,2,\ldots.$ So if $\sigma_{x}^{(\alpha,\gamma)}$ is a
Sheffer umbra whose moments are the Sheffer sequence $\{s_n(x)\},$ then the
moments $\{p_n(x)\}$ of the umbra $\sigma_{x}^{(\epsilon,\gamma)}$ represent
the sequence associated to $\{s_n(x)\}.$

In the following, we recall a generalization of the well-known Sheffer
identity, for the proof see \cite{DNS}.

\begin{theorem}[Generalized Sheffer identity] 
\label{gsi} A polynomial umbra $\sigma_{x}$ is a Sheffer umbra if and only
if there exists an umbra $\gamma,$ provided with a compositional inverse,
such that
\begin{equation}
\sigma_{\eta+\zeta} \equiv \sigma_{\eta} + \zeta \boldsymbol{.} \gamma^{*},
\label{(shefid)}
\end{equation}
for any $\eta, \zeta \in A.$
\end{theorem}

\begin{corollary}[The Sheffer identity] 
\label{(si)} A polynomial umbra $\sigma_{x}$ is a Sheffer umbra if and only
if there exists an umbra $\gamma,$ provided with a compositional inverse,
such that
\begin{equation}
\sigma_{x+y}\equiv\sigma_{x} + y \boldsymbol{.} \gamma^{*}.
\label{(shefid2)}
\end{equation}
\end{corollary}

Equivalence (\ref{(shefid2)}) gives the well-known Sheffer identity,
because, setting $s_{n}(x+y)=E[\sigma_{x+y}^{n}],$ $s_{k}(x)=E[%
\sigma_{x}^{k}]$ and $p_{n-k}(y)=E[(y\boldsymbol{.}\gamma^{*})^{n-k}]$ and
by using the binomial expansion, we have
\begin{equation}
s_{n}(x+y)=\sum_{k=0}^{n} {\binom{n }{k}} \,s_{k}(x)\,p_{n-k}(y).
\label{(shefid1)}
\end{equation}

\section{Solving recursions.}

In this section we show how to use the language of umbrae in order to solve
some recursions. The idea is to
characterize the umbrae $(\alpha, \gamma)$ in (\ref{(conrep)}) in such a way
that the umbra $\gamma$ is characterized by the recursion and the umbra $%
\alpha$ is characterized by the initial condition.

\subparagraph{Integral initial condition.} \label{1} Consider the
following difference equation
\begin{equation}
q_{n}(x+1)=q_{n}(x)+q_{n-1}(x)   \label{recurr1}
\end{equation}
under the condition $\int_{0}^{1}q_{n}(x)dx=1,$ for all nonnegative integers
$n$. We assume $q_n(x)=0$ for negative $n.$ Let $q_n(x) = s_n(x) / n!.$ Recursion (\ref{recurr1}) for $s_n(x)$
becomes
\begin{equation*}
s_{n}(x+1) = s_n(x) + n s_{n-1}(x),
\end{equation*}
for all nonnegative $n.$ This last recursion fits the Sheffer identity (\ref%
{(shefid2)}) if we set $y=1$ and choose the sequence $\{p_{n}(x)\}$ such
that $p_{0}(x)=1,p_{1}(1)=1$ and $p_{n}(1)=0$ for all $n \geq 2.$ This is
true if we choose $p_{n}(x) \simeq (x \boldsymbol{.} \chi)^n,$ that is the
sequence $\{p_{n}(x)\}$ associated to the umbra $u.$ Indeed, since we have
\begin{equation*}
\beta \boldsymbol{.} \chi \equiv \chi \boldsymbol{.} \beta \equiv u \quad %
\hbox{and} \quad \beta \boldsymbol{.} u^{\scriptscriptstyle <-1>} \equiv u^{%
\scriptscriptstyle <-1>} \boldsymbol{.} \beta \equiv \chi
\end{equation*}
(see also \cite{dinardoeurop}), the adjoint of the unity umbra $u$ is $u^{*}
\equiv\beta.u^{\scriptscriptstyle <-1>} \equiv\chi,$ and $x \boldsymbol{.}
\chi \equiv x \boldsymbol{.} u^{*}.$ So we are looking for solutions of (\ref%
{recurr1}) such that
\begin{equation}
q_n(x) \simeq \frac{(\alpha + x \boldsymbol{.} \chi)^n}{n!},
\label{recurr1bis}
\end{equation}
for some umbra $\alpha$ depending on the initial condition $%
\int_{0}^{1}q_{n}(x)dx=1.$ Since
\begin{equation*}
\int_{0}^{1} p(x) dx=E[p(-1 \boldsymbol{.} \iota)]
\end{equation*}
for all $p(x) \in R[x],$ where $- 1 \boldsymbol{.} \iota$ is the inverse of
the Bernoulli umbra (see \cite{Sheppard}), then
\begin{equation*}
\int_{0}^{1}q_{n}(x)dx=1 \Leftrightarrow E[s_n(- 1 \boldsymbol{.} \iota)]=n!
\quad\hbox{for all}\quad n \geq 1 \Leftrightarrow -1 \boldsymbol{.} \iota
\boldsymbol{.} \chi + \alpha \equiv \bar{u},
\end{equation*}
with $\bar{u}$ an umbra whose moments are $E[\bar{u}^n]=n!.$ Thus we have $%
\alpha \equiv \bar{u} + \iota \boldsymbol{.} \chi.$ Replacing this last
result in (\ref{recurr1bis}), solutions of (\ref{recurr1}) are moments of
the Sheffer umbra $\bar{u} + (\iota + x \boldsymbol{.} u) \boldsymbol{.} \chi
$ normalized by $n!,$
\begin{equation*}
q_n(x) \simeq \frac{[\bar{u} + (\iota + x \boldsymbol{.} u) \boldsymbol{.}
\chi]^n}{n!}.
\end{equation*}

\subparagraph{Pascal's recursions.}

Consider the following Pascal's recursion
\begin{equation}
q_{n}(x)=q_{n}(x-1)+q_{n-1}(x)  \label{recurr2}
\end{equation}%
that satisfies the initial condition
\begin{equation}
q_{n}(1-n)=\sum_{i=0}^{n-1}q_{i}(n-2\,i)\quad \hbox{for all}\quad n\geq
1,\,\,\,q_{0}(-1)=1.  \label{icr2}
\end{equation}%
Let $q_{n}(x)=s_{n}(x)/n!$ for all nonnegative $n.$ Recursion
(\ref{recurr2}) for $s_{n}(x)$ becomes
\begin{equation*}
s_{n}(x-1)=s_{n}(x)-ns_{n-1}(x).
\end{equation*}%
This last equation fits the Sheffer identity (\ref{(shefid2)}) if we set $%
y=-1$ and choose the sequence $\{p_{n}(x)\}$ such that $p_{0}(x)=1,$ $%
p_{1}(-1)=-1$ and $p_{n}(-1)=0$ for all $n\geq 2.$ From (\ref{(shefid)}), we
need to have $-1\boldsymbol{.}\gamma ^{\ast }\equiv -\chi \Leftrightarrow
\gamma ^{\ast }\equiv -1\boldsymbol{.}-\chi $ and so $p_{n}(x)\simeq (-x%
\boldsymbol{.}-\chi )^{n}.$ The umbra $-1\boldsymbol{.}-\chi $ is similar to
the umbra $\bar{u},$ defined in Example \ref{1} since
\begin{equation*}
f(-1\boldsymbol{.}-\chi ,t)=\frac{1}{1-t}=f(\bar{u},t).
\end{equation*}%
Then $p_{n}(x)\simeq (-x\boldsymbol{.}-\chi )^{n}\simeq (x\boldsymbol{.}\bar{%
u})^{n}$ and solutions of (\ref{recurr2}) are therefore of the type
\begin{equation}
q_{n}(x)\simeq \frac{(\alpha +x\boldsymbol{.}\bar{u})^{n}}{n!}.
\label{(ccc)}
\end{equation}%
Due to the initial condition (\ref{icr2}), we need to consider a
polynomial sequence such that
\begin{equation}
q_{n}(x-n)\simeq \frac{\lbrack \alpha +(x-n)\boldsymbol{.}\bar{u}]^{n}}{n!}%
\simeq \frac{\lbrack \alpha +(x-1)\boldsymbol{.}\chi ]^{n}}{n!}.
\label{(ccc1)}
\end{equation}
The equivalence on the right of (\ref{(ccc1)}) follows
from the following remarks. Since $f(x\boldsymbol{.}\bar{u}%
,t)=1/(1-t)^{x},$ we have $E[(x\boldsymbol{.}\bar{u})^{n}]=(x)^{n}=x(x+1)%
\cdots (x+n-1)=(x+n-1)_{n}$ and since $f(x\boldsymbol{.}\chi
,t)=(1+t)^{x},$ we have $E[(x\boldsymbol{.}\chi )^{n}]=(x)_{n}$
(see \cite{dinardoeurop}). Therefore,
\begin{equation*}
(x\boldsymbol{.}\bar{u})^{n}\simeq \lbrack (x+n-1)\boldsymbol{.}\chi
]^{n}\Leftrightarrow \lbrack (x-n)\boldsymbol{.}\bar{u}]^{n}\simeq \lbrack
(x-1)\boldsymbol{.}\chi ]^{n}
\end{equation*}
and if $\{s_{n}(x)\}$ is a Sheffer sequence with associated polynomials $(x%
\boldsymbol{.}\bar{u})^{n}/n!,$ then $\{s_{n}(x-n)\}$ is a Sheffer sequence
with associated polynomials $[(x-n)\boldsymbol{.}\bar{u}]^{n}/n!.$ Note that,
having written $q_{n}(x-n)$ as the latter equivalence in (\ref{(ccc1)}), the
values of $q_{n}(1-n)$ in the initial condition (\ref{icr2}) give exactly
the moments of $\alpha $ normalized by $n!.$ By observing that from the
latter equivalence in (\ref{(ccc1)})
\begin{equation*}
q_{n}(x)\simeq \frac{\lbrack \alpha +(x+n-1)\boldsymbol{.}\chi ]^{n}}{n!}%
\simeq \sum_{k=0}^{n}\frac{\alpha ^{k}}{k!}{\binom{{x+n-1}}{{n-k}}},
\end{equation*}
we have a first expression of $q_{n}(x),$ that is
\begin{equation*}
q_{n}(x)=\sum_{k=0}^{n}q_{k}(1-k){\binom{{x+n-1}}{{n-k}}}.
\end{equation*}%
From a computational point of view, this formula is very easy to
implement by using the recursion of the initial condition. One
could eliminate the contribution of this recursion in the
following way. By using the umbral expression of $q_{n}(x-n)$ we
have
\begin{equation}
\sum_{i=0}^{n-1}q_{i}(n-2i)=\sum_{i=0}^{n-1}q_{i}(x-i)|_{x=n-i}\simeq
\sum_{i=0}^{n-1}\frac{[\alpha +(n-i-1)\boldsymbol{.}\chi ]^{i}}{i!}
\label{ee1}
\end{equation}%
and from the initial condition we have
\begin{equation}
\frac{\alpha ^{n}}{n!}\simeq \sum_{j=0}^{n-1}\frac{\alpha ^{j}}{j!}%
\sum_{i=0}^{n-j-1}\frac{[(n-i-j-1)\boldsymbol{.}\chi ]^{i}}{i!}.  \label{ee2}
\end{equation}%
If we consider an umbra $\delta $ such that
\begin{equation*}
\frac{\delta ^{\,n}}{n!}\simeq \sum_{i=0}^{n}\frac{[(n-i)\boldsymbol{.}\chi
]^{i}}{i!},
\end{equation*}%
then equivalence (\ref{ee2}) gives
\begin{equation}
\alpha ^{n}\simeq n\,(\alpha +\delta )^{\,n-1},\quad \hbox{for all}\quad
n\geq 1.  \label{(ccc3)}
\end{equation}%
If we denote by $\gamma _{\scriptscriptstyle D}$ an auxiliary
umbra such that $\gamma _{\scriptscriptstyle D}^{n}\simeq \partial
_{\gamma }\gamma ^{n}\simeq n\gamma ^{n-1},$  then (\ref{(ccc3)})
can be written as $\alpha \equiv (\alpha +\delta )_{\scriptscriptstyle D}.$ Since $%
f(\gamma _{\scriptscriptstyle D},t)=1+tf(\gamma ,t)$ then
$f(\alpha ,t)=1+tf(\alpha ,t)f(\delta ,t)$ gives $f(\alpha
,t)=1/[1-tf(\delta ,t)]$, and in umbral terms $\alpha \equiv \bar{u}\boldsymbol{.}\beta \boldsymbol{%
.}\delta _{{\scriptscriptstyle D}}.$ The solution of recurrence (\ref%
{recurr2}) is therefore
\begin{equation*}
q_{n}(x)\simeq \frac{\lbrack \bar{u}\boldsymbol{.}\beta \boldsymbol{.}\delta
_{{\scriptscriptstyle D}}+(x+n-1)\boldsymbol{.}\chi ]^{n}}{n!}.
\end{equation*}%

\begin{remark}
\rm{We call the umbra $\delta$ in the previous example 
\textit{Fibonacci umbra}. Indeed, because
\begin{equation*}
\delta^{\,n} \simeq n! \sum_{k=0}^{n} \frac{(k \boldsymbol{.} \chi)^{n-k}}{(n-k)!}
\simeq \sum_{k \geq 0} {\binom{n }{k}} \bar{u}^{k} (k \boldsymbol{.}
\chi)^{n-k},
\end{equation*}
we have $\delta \equiv \bar{u} \boldsymbol{.} \beta \boldsymbol{.} \chi_{{%
\scriptscriptstyle D}},$ by using (\ref{(gr:3ter)}), with $x$ replaced by $%
\bar{u},$ and Lemma 15 of \cite{dinardoeurop}. Thus we have
\begin{equation*}
f(\delta,t)=\frac{1}{1-t(1+t)}=\frac{1}{1-t-t^{2}}=\sum_{n \geq 0} F_n t^n,
\end{equation*}
which is the generating function of Fibonacci numbers $\{F_n\}$ with $F_0=1,$
so that $\delta^n \simeq n! F_n.$}
\end{remark}

\subparagraph{Fibonacci's numbers.} Consider the following
difference equation $F_{n}(m)=F_{n}(m-1)+F_{n-1}(m-2)$ under the
condition $F_{n}(0)=1$ for all nonnegative integers $n,$ looking
for a polynomial extension. Replace $m$ with $x+n+1$. Then the
difference equation can be rewritten as
\begin{equation}
F_{n}(x+n+1)=F_{n}(x+n)+F_{n-1}(x+n-1).  \label{recurr3bis}
\end{equation}%
Let $F_{n}(x+n)=s_{n}(x)/n!$ for all nonnegative $n.$ Recursion (\ref%
{recurr3bis}) for $s_{n}(x)$ becomes
\begin{equation*}
s_{n}(x+1)=s_{n}(x)+ns_{n-1}(x).
\end{equation*}%
This last recursion fits the Sheffer identity (\ref{(shefid2)}) if we set $%
y=1$ and choose the sequence $\{p_{n}(x)\}$ such that $p_{0}(x)=1,p_{1}(1)=1$
and $p_{n}(1)=0$ for all $n\geq 2.$ As in the previous paragraph, we are
looking for solutions of (\ref{recurr3bis}) such that
\begin{equation*}
F_{n}(x+n)\simeq \frac{(\alpha +x\boldsymbol{.}\chi )^{n}}{n!}.
\end{equation*}%
Let us observe that equation (\ref{recurr3bis}), for $x=0,$ gives the
well-known recurrence relation for Fibonacci numbers so that $%
n!F_{n}(n)\simeq \delta ^{n},$ where $\delta $ is the Fibonacci
umbra defined in the previous paragraph. Because $n!F_{n}(n)\simeq
\alpha ^{n}$ we have $\alpha \equiv \delta .$ Solutions of
(\ref{recurr3bis}) are such that
\begin{equation*}
F_{n}(x+n)\simeq \frac{(\delta +x\boldsymbol{.}\chi )^{n}}{n!}.
\end{equation*}%
Recalling $\delta \equiv \bar{u}\boldsymbol{.}\beta \boldsymbol{.}\chi _{{%
\scriptscriptstyle D}}$ we find
\begin{equation*}
F_{n}(x+n)\simeq \sum_{k=0}^{n}\frac{(x\boldsymbol{.}\chi )^{n-k}}{(n-k)!}%
\sum_{j=0}^{k}\frac{(j\boldsymbol{.}\chi )^{k-j}}{(k-j)!}\simeq
\sum_{k=0}^{n}\frac{[(x+k)\boldsymbol{.}\chi ]^{n-k}}{(n-k)!}\simeq
\sum_{k=0}^{n}{\binom{{x+k}}{{n-k}}},
\end{equation*}%
by which we can verify that the initial conditions
$F_{n}(0)=F_{n}(-n+n)=1$ hold.

\medskip

In the following example, the recursion does not fit the Sheffer
identity (\ref{(shefid1)}) directly, but we are able to translate
the recursion in an umbral equivalence by which the umbra $\gamma
$ can be still characterized.

\subparagraph{Dyck paths.} A ballot path takes up steps (\textit{u}) and
right steps (\textit{r}), starting at the origin and staying weakly above
the diagonal. For example, \textit{ururuur} is a ballot path to
$\left( 3,4\right).$ We denote by $D\left( n,m\right)$ the number
of ballot paths to $\left( n,m\right)$ under the conditions

\begin{enumerate}
\item no path goes below the diagonal (Dyck path);

\item no path contains the pattern (substring) \textit{urru} (see
Table 1).

\end{enumerate}
\begin{table}[ht]
\centering{\begin{tabular}{c||cccccc}
$m$ & 1 & 7 & 22 & 46 & 82 & 132 \\
6 & 1 & 6 & 16 & 29 & 46 & 63 \\
5 & 1 & 5 & 11 & 17 & 23 & 23 \\
4 & 1 & 4 & 7 & 9 & 9 &  \\
3 & 1 & 3 & 4 & 4 &  &  \\
2 & 1 & 2 & 2 &  &  &  \\
1 & 1 & 1 &  &  &  &  \\
0 & 1 &  &  &  &  &  \\ \hline\hline
& 0 & 1 & 2 & 3 & 4 & $n$%
\end{tabular}}
\caption{The ballot path}
\end{table}

The numbers $D\left( n,m\right) $ follow the difference equation
\begin{equation}
D(n,m)-D(n-1,m)=D(n,m-1)-D(n-2,m-1)+D(n-3,m-1)  \label{ccc4}
\end{equation}%
for all nonnegative $n$ and $m,$ under the initial condition $%
D(n,n)=D(n-1,n),$ and $D\left( 0,0\right) =1.$ Let $D(n,m)=s_{n}(m)/n!.$ The
recursion (\ref{ccc4}) for $s_{n}(m)$ becomes
\begin{equation}
s_{n}(m)-ns_{n-1}(m)=s_{n}(m-1)-\left( n\right) _{2}s_{n-2}(m-1)+\left(
n\right) _{3}s_{n-3}(m-1),  \label{ccc5}
\end{equation}%
and the initial condition becomes $s_{n}(n)=ns_{n-1}(n).$ We are looking for
solutions of Sheffer type $E[\sigma _{x}^{n}]\simeq s_{n}(x),$ i.e. $\sigma
_{x}\equiv \alpha +x\boldsymbol{.}\gamma ^{\ast }.$ Because $%
s_{n}(x)-ns_{n-1}(x)\simeq (\sigma _{x}-\chi )^{n},$ we can replace $\sigma
_{x}$ by $\sigma _{x-1}+\gamma ^{\ast }$ from Sheffer identity (\ref%
{(shefid2)}) with $y=-1.$ Comparing $(\sigma _{x}-\chi )^{n}\simeq
(\sigma _{x-1}+\gamma ^{\ast }-\chi )^{n}$ with the recursion
(\ref{ccc5}), with $x$ replaced by $m,$ a characterization of the
moments $(\gamma ^{\ast })^{n}$ is available. Indeed we have
\begin{equation*}
(\sigma _{x-1}+\gamma ^{\ast }-\chi )^{n}\simeq \sigma
_{x-1}^{n}+n\sum_{i=0}^{n-1}{\binom{{n-1}}{i}}\sigma _{x-1}^{n-i-1}\left[
\frac{(\gamma ^{\ast })^{i+1}}{i+1}-(\gamma ^{\ast })^{i}\right] .
\end{equation*}%
By replacing $x$ with $m$ in the previous equivalence and by applying the
linear functional $E$ to both sides, we have
\begin{equation}
s_{n}(m)-ns_{n-1}(m)\simeq s_{n}(m-1)+n\sum_{i=0}^{n-1}{\binom{{n-1}}{i}}%
s_{n-i-1}(m-1)\left[ \frac{g_{i+1}}{i+1}-g_{i}\right] ,  \label{ccc6}
\end{equation}%
where $E[(\gamma ^{\ast })^{i}]=g_{i}.$ By comparing (\ref{ccc6}) with (\ref%
{ccc5}), we have $g_{0}=1,g_{1}=1,g_{2}=0,g_{k}=k!$ for $k\geq 3.$ Then we
have
\begin{equation*}
f(\gamma ^{\ast },t)=1+t+\sum_{k\geq 3}t^{k}=\frac{1-t^{2}+t^{3}}{1-t}%
\Rightarrow f(x\boldsymbol{.}\gamma ^{\ast },t)=\left[ \frac{1-t^{2}+t^{3}}{%
1-t}\right] ^{x}.
\end{equation*}%
In order to characterize the umbra $\alpha $ we need to use the initial
condition on $s_{n}(x),$ that is $s_{n}(n)=ns_{n-1}(n),$ in umbral terms
\begin{equation}
(\alpha +n\boldsymbol{.}\gamma ^{\ast })^{n}\simeq n(\alpha +n\boldsymbol{.}%
\gamma ^{\ast })^{n-1}.  \label{ccc7bis}
\end{equation}%
Suppose $\alpha \equiv \bar{u}+\zeta .$ In this case the initial condition (%
\ref{ccc7bis}) gives
\begin{equation*}
(\zeta +n\boldsymbol{.}\gamma ^{\ast })^{n}\simeq \epsilon^{n}
\end{equation*}%
for all nonnegative integers $n.$ By applying the binomial expansion, we
have
\begin{equation}
(n\boldsymbol{.}\gamma ^{\ast })^{n}\simeq \sum_{k=1}^{n}{\binom{n}{k}}%
(-\zeta ^{k})(n\boldsymbol{.}\gamma ^{\ast })^{n-k}.  \label{ccc7}
\end{equation}
The following proposition allows us to take a step forward.
\begin{proposition}
\textrm{\label{reci} If $\gamma$ is an umbra with $E[\gamma]=1,$ then for $n
\geq 1$ we have $(x \boldsymbol{.} \gamma^*)^n \simeq x (\eta + x
\boldsymbol{.} \gamma^*)^{n-1},$ with $\eta$ an umbra such that $\eta^n
\simeq (\gamma^{{\scriptscriptstyle <-1>}})^{n+1}$ for $n \geq 1.$ }
\end{proposition}
\begin{proof}
Compare the generating function of the umbra $\eta + x \punt \gamma^*$ with the
one obtained by replacing $(\eta + x \punt \gamma^*)^{n}$ with $(x
\punt \gamma^*)^{n+1} / x.$
\end{proof}
Due to the previous Proposition, with $x$ replaced by $n,$ we have
\begin{equation}
(n \boldsymbol{.} \gamma^*)^n \simeq n (\eta + n \boldsymbol{.}
\gamma^*)^{n-1} \simeq \sum_{k=1}^n {\binom{n }{k}} k [\gamma^{{%
\scriptscriptstyle <-1>}}]^k (n \boldsymbol{.} \gamma^*)^{n-k}.  \label{ccc8}
\end{equation}
By comparing equivalence (\ref{ccc8}) with equivalence (\ref{ccc7}), we have
\begin{equation*}
\zeta^k \simeq - k \, \, \eta^{k-1} \simeq - k [\gamma^{{\scriptscriptstyle %
<-1>}}]^k
\end{equation*}
for $k \geq 1.$ Then, the solution of the recurrence relation (\ref{ccc5})
is
\begin{equation*}
\frac{s_n(x)}{n!} \simeq \frac{(\bar{u} + \zeta + x \boldsymbol{.}
\gamma^*)^n}{n!}.
\end{equation*}
Let us observe that
\begin{equation*}
s_n(x) \simeq \sum_{k=0}^n {\binom{n }{k}} (n-k)! (\zeta + x \boldsymbol{.}
\gamma^*)^k \simeq \sum_{k=0}^n {\binom{n }{k}} (n-k)! {\frac{{x-k} }{x}}(x
\boldsymbol{.} \gamma^*)^k
\end{equation*}
because
\begin{equation*}
x \sum_{j=1}^k {\binom{k }{j}} j \eta^{j-1} (x \boldsymbol{.}
\gamma^*)^{k-j} \simeq x \, k \, (\eta + x \boldsymbol{.} \gamma^*)^{k-1}
\simeq k (x \boldsymbol{.} \gamma^*)^k,
\end{equation*}
due to Proposition \ref{reci}. 
More on counting paths containing a given number of a certain pattern can be found in
\cite{Sapounakis} and \cite{Nie5}. 

\medskip 

We conclude this section stating a generalization of Proposition \ref{reci} which turns out to be
useful in solving the class of recursions involving Sheffer
sequences satisfying the initial condition $s_n(-c \,n)=\delta_{0,n}.$
\begin{theorem}
If $\gamma$ is an umbra with $E[\gamma]=1,$ then for $n \geq 1$ we have $x
(\chi \boldsymbol{.} c \boldsymbol{.} \beta \boldsymbol{.} \eta_{%
\scriptscriptstyle D} + x \boldsymbol{.} \gamma^*)^n \simeq (x + c \, n) (x
\boldsymbol{.} \gamma^*)^n,$ with $\eta$ an umbra such that $\eta^n \simeq
(\gamma^{{\scriptscriptstyle <-1>}})^{n+1}$ for $n \geq 1$ and $c \in R.$
\end{theorem}
\begin{proof}
Due to Proposition \ref{reci}, we have
$$\frac{(x+cn)}{x} (x \punt \gamma^{*})^n \simeq (x \punt \gamma^*)^n + c n (\eta + x \punt \gamma^*)^{n-1} 
 \simeq  (x \punt \gamma^*)^n + c \sum_{j=1}^{n} {n \choose j} j \eta^{j-1} (x \punt \gamma^*)^{n-j}.$$
The result follows by observing that $j \, \eta^{j-1} \simeq \eta_{\scriptscriptstyle D}^{j}$ and
$c \simeq (\chi \punt c \punt \beta)^j,$ for $j=1, 2, \ldots, n,$ so that
$c \, j \, \eta^{j-1} \simeq (\chi \punt c \punt \beta)^j \eta_{\scriptscriptstyle D}^{j}
\simeq [(\chi \punt c \punt \beta)\eta_{\scriptscriptstyle D}]^j \simeq (\chi \punt c \punt \beta \punt \eta_{\scriptscriptstyle D})^j.$
\end{proof}
If $\gamma$ is the umbra characterized by a linear recursion and $s_n(-c \,
n)=\delta_{0,n}$ is the initial condition, then the previous proposition
states that solutions have the form $(\chi \boldsymbol{.} c \boldsymbol{.}
\beta \boldsymbol{.} \eta_{\scriptscriptstyle D} + x \boldsymbol{.}
\gamma^*)^n.$

\end{document}